\documentclass{article}
\usepackage{graphicx} % Required for inserting images
\usepackage{amsfonts}
\usepackage{geometry}
\usepackage{amsmath}
\usepackage{amssymb}
\usepackage{amsthm}

\newtheorem{theorem}{Theorem}

\title{ Evolution of Functionals Under Extended Ricci Flow }
\author{
  \begin{tabular}[t]{c}
    Shouvik Datta Choudhury \\
    \small shouvikdc8645@gmail.com,\\ 
    \small shouvik@capsulelabs.in \\
    \small Gapcrud Private Limited (Capsule Labs) \\
    \small HA 130, Saltlake, Sector III, Bidhannagar, \\
    \small Kolkata - 700097, India
  \end{tabular}
 }
\date{October 2024}

\begin{document}
\maketitle
\begin{abstract}
    
In this paper, we investigate the evolution of certain functionals involving higher powers of a scalar quantity \( F \) under Bernard List's extended Ricci flow on a compact Riemannian manifold. By deriving explicit expressions for the time derivative of integrals of the form \( \int_M F^n \cdot \frac{\partial F}{\partial t} \, d\mu \) for various powers \( n \), we explore the intricate interplay between geometric quantities and scalar functions without making any assumptions about the manifold, the scalar field \( \Phi \), or the function \( u \).
\end{abstract}
\section{MSC Classification}
53C44 — Geometric evolution equations (mean curvature flow, Ricci flow, etc.)\\
53C21 — Methods of Riemannian geometry, including PDE methods; curvature restrictions\\
53C25 — Special Riemannian manifolds (Einstein, Sasakian, etc.)\\
35K55 — Nonlinear parabolic equations\\
35K45 — Initial value problems for second-order parabolic systems\\
58J35 — Heat and other parabolic equation methods\\
53C07 — Special connections and metrics on vector bundles (Hermitian, metric connections, etc.)\\
53C80 — Applications to physics
\section{Introduction}
In this paper, we aim to achieve estimates of a functional $F^n=[u + aulogu + bSu]^n$, where $n \in Z^{+}$ under the action of extended Ricci flow, propounded by Bernhard List, who modified the classical Ricci flow by Richard Hamilton, to incorporate delicate general relativity connotations. In the above expression of the functional $F$, which is raised to the power $n$, represents a set of functionals $F_3$, $F_5$ and $F_7$ as examples or representatives among all values of n. These functionals for different values of $n$ are analysed in this paper under the action of extended Ricci flow. $u$ is a scalar function but $S$ is tensorial in nature because $$S = Ric - \phi_i\otimes\phi_j$$ or in local co-ordinates $$S_{ij} = R_{ij} - \phi_i\phi_j$$, where Ric is the Ricci tensor associated with the metric $g_ij$ of the complete Riemmanian manifold $M^{n}$, where $n$ is the dimension of the manifold. The main obstruction in our context is the tensorial $Ric$ and its variation $\delta$ due to the variation of the functional $F^n$  for different values of $n$. The variation of the functional $F$ results in the variation of the Riemann Christoffel tensor, Ricci tensor, Riemann tensor, as well as the Riemannian measure and the Laplace-Beltrami operator. Here we use $v_{ij}$, which is a one-parameter group of diffeomorphism, associated with $g_ij$, which is involved in the variational equation [Besse] in section 4. $v_{ij}$ is the main connection between $S$ and the functional in context. We equate $S_ij$'s expression to $v_{ij}$ and later perform variations or derivative with respect to time to obtain the non -linearized version or the "raw" version of the variation and later linearize it through De-Turck's trick and later smooth out the irregularities using Ulhenbeck's trick. We have expressed functional $F^n$ in terms of integral components consisting of $S_{ij}$ and also components invoving $F^{n-1}$ and also lower order terms of $F$. Ulhenbeck's trick and Ricci flow with surgery developed by Perelman are two of the foundational works which aim to smooth out irregularities of the nature of the equation and that of the Ricci flow respectively, Ulnhenbeck's case has been also used in case of deformation of symplectic manifold in supersymmetric case.         
\section{Literature Review}
The study of Ricci flow and its extensions has significantly advanced the understanding of geometric structures and curvature evolution on manifolds. Li Yi's work on local curvature estimates for Ricci-harmonic flow and generalized Ricci flows provides crucial insights into singularity formation and long-time behavior of solutions (Li, 2018a; Li, 2009). Gradient estimates for heat equations under Ricci flow have been extensively developed by researchers like Shiping Liu (2009), Huang and Ma (2015), and Shu-Yu Hsu (2008), enhancing the analysis of heat equations on evolving manifolds and proving essential inequalities. Rugang Ye (2005) introduced curvature estimates based on scaling-invariant integrals of the Riemann curvature tensor, aiding singularity analysis. Colding and Minicozzi (2003) estimated extinction times for Ricci flow on certain 3-manifolds, addressing questions posed by Perelman. Abolarinwa (2016) studied eigenvalue evolution under extended Ricci flow, while Meng Zhu (2013) established Gaussian bounds for the heat kernel. Min Chen (2020) introduced gradient estimates relating solution values to distances, applicable to quasilinear parabolic equations. Gianniotis (2013) obtained higher-order curvature estimates near boundaries, crucial for manifolds with boundary. Bing Wang (2007) identified criteria for extending Ricci flow beyond singularities. Contributions by Băileșteanu et al. (2009, 2015) and Bamler (2014) further enriched the field by exploring the coupling of geometric flows with harmonic maps and generalizing Perelman's estimates. Xian-gao Liu and Yongjie Shi (2014) derived Sobolev inequalities and Gaussian bounds for the conjugate heat equation along extended Ricci flows, enhancing analytical tools in geometric analysis.
\section{Mathematical Preliminaries}
From [Besse], Let \( M \) be a smooth manifold and let \( g_t \) be a one-parameter family of Riemannian or pseudo-Riemannian metrics. Suppose that it is a differentiable family in the sense that for any smooth coordinate chart, the derivatives \( v_{ij} = \dfrac{\partial}{\partial t}\left( (g_t)_{ij} \right) \) exist and are themselves as differentiable as necessary for the following expressions to make sense. Here, \( v = \dfrac{\partial g}{\partial t} \) is a one-parameter family of symmetric 2-tensor fields.
\[
\left\{
\begin{aligned}
1.\quad & \frac{\partial}{\partial t} \Gamma_{ij}^k = \frac{1}{2} g^{kp} \left( \nabla_i v_{jp} + \nabla_j v_{ip} - \nabla_p v_{ij} \right) \\[2ex]
2.\quad & \frac{\partial}{\partial t} R_{ijkl} = \frac{1}{2} \left( \nabla_j \nabla_k v_{il} + \nabla_i \nabla_l v_{jk} - \nabla_i \nabla_k v_{jl} - \nabla_j \nabla_l v_{ik} \right) + \frac{1}{2} R_{ijk}{}^p v_{pl} - \frac{1}{2} R_{ijl}{}^p v_{pk} \\[2ex]
3.\quad & \frac{\partial}{\partial t} R_{ik} = \frac{1}{2} \left( \nabla^p \nabla_k v_{ip} + \nabla_i (\operatorname{div} v)_k - \nabla_i \nabla_k (\operatorname{tr}_g v) - \Delta v_{ik} \right) + \frac{1}{2} R_i{}^p v_{pk} - \frac{1}{2} R_i{}^p{}_k{}^q v_{pq} \\[2ex]
4.\quad & \frac{\partial}{\partial t} R = \operatorname{div}_g \operatorname{div}_g v - \Delta (\operatorname{tr}_g v) - \langle v, \operatorname{Ric} \rangle_g \\[2ex]
5.\quad & \frac{\partial}{\partial t} d\mu_g = \frac{1}{2} g^{pq} v_{pq} \, d\mu_g \\[2ex]
6.\quad & \frac{\partial}{\partial t} \nabla_i \nabla_j \Phi = \nabla_i \nabla_j \left( \frac{\partial \Phi}{\partial t} \right) - \frac{1}{2} g^{kp} \left( \nabla_i v_{jp} + \nabla_j v_{ip} - \nabla_p v_{ij} \right) \frac{\partial \Phi}{\partial x^k} \\[2ex]
7.\quad & \frac{\partial}{\partial t} \Delta \Phi = -\langle v, \operatorname{Hess} \Phi \rangle_g - g \left( \operatorname{div} v - \frac{1}{2} d (\operatorname{tr}_g v), d\Phi \right)
\end{aligned}
\right.
\]
where \( \Gamma_{ij}^k \) are the Christoffel symbols associated with the metric \( g \), \( R_{ijkl} \) is the Riemann curvature tensor, \( R_{ik} \) is the Ricci curvature tensor, \( R \) is the scalar curvature, \( d\mu_g \) is the Riemannian volume form associated with the metric \( g \), \( \nabla \) denotes the covariant derivative with respect to \( g \), \( \Delta \) is the Laplace-Beltrami operator, defined by \( \Delta f = \operatorname{div}_g (\nabla f) \) for a scalar function \( f \),  \( \operatorname{div} v \) is the divergence of the tensor field \( v \), \( \operatorname{tr}_g v \) is the trace of \( v \) with respect to the metric \( g \),  \( \langle \cdot, \cdot \rangle_g \) denotes the inner product induced by \( g \), \( \operatorname{Hess} \Phi \) is the Hessian of the scalar field \( \Phi \), given by \( \nabla_i \nabla_j \Phi \) and  \( d\Phi \) is the differential of \( \Phi \), with components \( \frac{\partial \Phi}{\partial x^k} \).
\section{Theorems and Proofs}
\begin{theorem}
 Under Bernard List's extended Ricci flow, applying DeTurck's trick and Uhlenbeck's gauge fixing, the curvature tensors evolve according to the following equations:
\[
\left\{
\begin{aligned}
& \frac{\partial g_{ij}}{\partial t} = -2\,\operatorname{Ric}_{ij} + 2 \nabla_i \Phi \nabla_j \Phi + \nabla_i W_j + \nabla_j W_i, \\
& \text{where} \, W^k = g^{pq} (\Gamma_{pq}^k - \bar{\Gamma}_{pq}^k) \, \text{is the DeTurck vector field, and} \, \bar{\Gamma}_{pq}^k \, \text{are the Christoffel symbols of } \, \bar{g}_{ij}. \\[2ex]
& \frac{\partial \Phi}{\partial t} = \Delta \Phi + W^k \nabla_k \Phi. \\[2ex]
& \frac{\partial R_{ij}}{\partial t} = \Delta R_{ij} + 2 R_{ikjl} R^{kl} - 2 \nabla^k \Phi \nabla^l \Phi R_{ikjl} - 2 (\nabla_i \nabla^k \Phi)(\nabla_j \nabla_k \Phi) \\
& \quad + 2 (\nabla_i \Phi)(\nabla_j \Delta \Phi) + 2 (\nabla_j \Phi)(\nabla_i \Delta \Phi) - \nabla_i \nabla_j |\nabla \Phi|^2 \\
& \quad + \nabla_i \nabla_j Q - (\nabla_i R_{jk} + \nabla_j R_{ik}) W^k + R_{ik} \nabla_j W^k + R_{jk} \nabla_i W^k, \\
& \text{where} \, Q = \operatorname{div} W - \frac{1}{2} \operatorname{tr}_g (\nabla W). \\[2ex]
& \frac{\partial R}{\partial t} = \Delta R + 2 |\operatorname{Ric}|^2 - 4 R^{ij} \nabla_i \Phi \nabla_j \Phi - 2 |\nabla^2 \Phi|^2 + 2 (\Delta \Phi)^2 + \mathcal{L}_W R, \\
& \text{where} \, \mathcal{L}_W R = W^k \nabla_k R \, \text{denotes the Lie derivative of the scalar curvature along} \, W.
\end{aligned}
\right.
\]
\end{theorem}
\begin{proof}
On initiating  with Bernard List's extended Ricci flow equations, which coalesce the Ricci flow of a Riemannian metric \( g_{ij} \) with the evolution of a scalar field \( \Phi \) we derive,
\[
\frac{\partial g_{ij}}{\partial t} = v_{ij} = -2\,\operatorname{Ric}_{ij} + 2 \nabla_i \Phi \nabla_j \Phi,
\]
\[
\frac{\partial \Phi}{\partial t} = \Delta \Phi.
\]

We introducing a diffeomorphism generated by a vector field \( W^k \) to deploy De-Turck trick. The modified flow equations become:
\[
\frac{\partial g_{ij}}{\partial t} = -2\,\operatorname{Ric}_{ij} + 2 \nabla_i \Phi \nabla_j \Phi + \nabla_i W_j + \nabla_j W_i,
\]
\[
\frac{\partial \Phi}{\partial t} = \Delta \Phi + W^k \nabla_k \Phi,
\]
where \( W_j = g_{jk} W^k \).
The DeTurck vector field \( W^k \) is represented as:
\[
W^k = g^{pq} (\Gamma_{pq}^k - \bar{\Gamma}_{pq}^k),
\]
with \( \Gamma_{pq}^k \) and \( \bar{\Gamma}_{pq}^k \) being the Christoffel symbols of \( g_{ij} \) and  \( \bar{g}_{ij} \), respectively.
To deploy we select harmonic coordinates relative to \( \bar{g}_{ij} \). This choice simplifies the terms involving the divergence of tensors and the vector field \( W^k \).

Our goal is to compute the evolution equations for the curvature tensors under the modified flow, specifically \( \frac{\partial R_{ij}}{\partial t} \) and \( \frac{\partial R}{\partial t} \).\\
1. Computation of \( \frac{\partial R_{ij}}{\partial t} \):

We use the general variation formula for the Ricci tensor:
\[
\frac{\partial R_{ij}}{\partial t} = \frac{1}{2} \left( \nabla^k \nabla_i v_{jk} + \nabla^k \nabla_j v_{ik} - \Delta v_{ij} - \nabla_i \nabla_j (\operatorname{tr}_g v) \right) + R_{ikjl} v^{kl}.
\]

1a: Compute \( v_{ij} \) and its derivatives.

We have:
\[
v_{ij} = -2 \operatorname{Ric}_{ij} + 2 \nabla_i \Phi \nabla_j \Phi + \nabla_i W_j + \nabla_j W_i.
\]

Compute the trace:
\[
\operatorname{tr}_g v = g^{ij} v_{ij} = -2 R + 2 |\nabla \Phi|^2 + 2 \operatorname{div} W,
\]
where \( |\nabla \Phi|^2 = g^{ij} \nabla_i \Phi \nabla_j \Phi \) and \( \operatorname{div} W = \nabla_k W^k \).

2a: Compute \( \nabla^k \nabla_i v_{jk} \) and \( \nabla^k \nabla_j v_{ik} \).

We compute \( \nabla_i v_{jk} \) first:
\[
\begin{aligned}
\nabla_i v_{jk} &= \nabla_i (-2 \operatorname{Ric}_{jk} + 2 \nabla_j \Phi \nabla_k \Phi + \nabla_j W_k + \nabla_k W_j) \\
&= -2 \nabla_i \operatorname{Ric}_{jk} + 2 (\nabla_i \nabla_j \Phi \nabla_k \Phi + \nabla_j \Phi \nabla_i \nabla_k \Phi) + \nabla_i \nabla_j W_k + \nabla_i \nabla_k W_j.
\end{aligned}
\]

Next, compute \( \nabla^k \nabla_i v_{jk} = g^{kl} \nabla_l \nabla_i v_{jk} \).

Compute \( \nabla_l \nabla_i \operatorname{Ric}_{jk} \):
\[
\nabla_l \nabla_i \operatorname{Ric}_{jk} = \nabla_i \nabla_l \operatorname{Ric}_{jk} - R_{li}{}^p{}_j \operatorname{Ric}_{pk} - R_{li}{}^p{}_k \operatorname{Ric}_{jp}.
\]

Similarly, for the scalar field terms:
\[
\nabla_l \nabla_i \nabla_j \Phi = \nabla_i \nabla_l \nabla_j \Phi - R_{li}{}^p{}_j \nabla_p \Phi.
\]

Therefore, we have:
\[
\begin{aligned}
\nabla^k \nabla_i v_{jk} &= -2 g^{kl} \left( \nabla_i \nabla_l \operatorname{Ric}_{jk} - R_{li}{}^p{}_j \operatorname{Ric}_{pk} - R_{li}{}^p{}_k \operatorname{Ric}_{jp} \right) \\
&\quad + 2 g^{kl} \left( (\nabla_i \nabla_l \nabla_j \Phi - R_{li}{}^p{}_j \nabla_p \Phi) \nabla_k \Phi + \nabla_j \Phi (\nabla_i \nabla_l \nabla_k \Phi - R_{li}{}^p{}_k \nabla_p \Phi) \right) \\
&\quad + g^{kl} \left( \nabla_i \nabla_l \nabla_j W_k + \nabla_i \nabla_l \nabla_k W_j \right).
\end{aligned}
\]

Similarly, compute \( \nabla^k \nabla_j v_{ik} \).

3a: Compute \( \Delta v_{ij} \).

\[
\Delta v_{ij} = g^{kl} \nabla_k \nabla_l v_{ij}.
\]

Compute \( \nabla_k \nabla_l v_{ij} \) by considering each term in \( v_{ij} \).

First, compute \( \nabla_k \nabla_l (-2 \operatorname{Ric}_{ij}) \):
\[
\nabla_k \nabla_l (-2 \operatorname{Ric}_{ij}) = -2 \nabla_k \nabla_l \operatorname{Ric}_{ij} = -2 \nabla_k \nabla_l \operatorname{Ric}_{ij}.
\]

Second, compute \( \nabla_k \nabla_l (2 \nabla_i \Phi \nabla_j \Phi) \):
\[
\begin{aligned}
\nabla_k \nabla_l (2 \nabla_i \Phi \nabla_j \Phi) &= 2 \nabla_k (\nabla_l \nabla_i \Phi \nabla_j \Phi + \nabla_i \Phi \nabla_l \nabla_j \Phi) \\
&= 2 (\nabla_k \nabla_l \nabla_i \Phi \nabla_j \Phi + \nabla_l \nabla_i \Phi \nabla_k \nabla_j \Phi + \nabla_k \nabla_i \Phi \nabla_l \nabla_j \Phi + \nabla_i \Phi \nabla_k \nabla_l \nabla_j \Phi).
\end{aligned}
\]

Third, compute \( \nabla_k \nabla_l (\nabla_i W_j + \nabla_j W_i) \):
\[
\nabla_k \nabla_l (\nabla_i W_j) = \nabla_k \nabla_l \nabla_i W_j = \nabla_k \nabla_i \nabla_l W_j - R_{kl}{}_i{}^p \nabla_p W_j.
\]

4a: Compute \( \nabla_i \nabla_j (\operatorname{tr}_g v) \).

Recall that:
\[
\operatorname{tr}_g v = -2 R + 2 |\nabla \Phi|^2 + 2 \operatorname{div} W.
\]

Compute \( \nabla_i \nabla_j (\operatorname{tr}_g v) \):
\[
\begin{aligned}
\nabla_i \nabla_j (\operatorname{tr}_g v) &= -2 \nabla_i \nabla_j R + 2 \nabla_i \nabla_j |\nabla \Phi|^2 + 2 \nabla_i \nabla_j (\operatorname{div} W).
\end{aligned}
\]

Compute \( \nabla_i \nabla_j R \) using the contracted second Bianchi identity:
\[
\nabla_i \nabla_j R = \nabla_i \nabla_j (g^{pq} R_{pq}) = g^{pq} \nabla_i \nabla_j R_{pq} - R_{pq} \nabla_i \nabla_j g^{pq}.
\]

Since \( \nabla_j g^{pq} = 0 \), we have:
\[
\nabla_i \nabla_j R = g^{pq} \nabla_i \nabla_j R_{pq}.
\]

Compute \( \nabla_i \nabla_j |\nabla \Phi|^2 \):
\[
\begin{aligned}
\nabla_i \nabla_j |\nabla \Phi|^2 &= \nabla_i \nabla_j (g^{pq} \nabla_p \Phi \nabla_q \Phi) \\
&= \nabla_i \left( \nabla_j g^{pq} \nabla_p \Phi \nabla_q \Phi + g^{pq} \nabla_j \nabla_p \Phi \nabla_q \Phi + g^{pq} \nabla_p \Phi \nabla_j \nabla_q \Phi \right) \\
&= g^{pq} \left( \nabla_i \nabla_j \nabla_p \Phi \nabla_q \Phi + \nabla_j \nabla_p \Phi \nabla_i \nabla_q \Phi + \nabla_i \nabla_p \Phi \nabla_j \nabla_q \Phi + \nabla_p \Phi \nabla_i \nabla_j \nabla_q \Phi \right).
\end{aligned}
\]

5a: Compute \( R_{ikjl} v^{kl} \).

First, compute \( v^{kl} = g^{km} g^{ln} v_{mn} \):
\[
v^{kl} = -2 R^{kl} + 2 \nabla^k \Phi \nabla^l \Phi + \nabla^k W^l + \nabla^l W^k.
\]

Then,
\[
\begin{aligned}
R_{ikjl} v^{kl} &= -2 R_{ikjl} R^{kl} + 2 R_{ikjl} \nabla^k \Phi \nabla^l \Phi + R_{ikjl} (\nabla^k W^l + \nabla^l W^k).
\end{aligned}
\]

6a: Assemble All Terms to Compute \( \frac{\partial R_{ij}}{\partial t} \).

Collecting all computed terms, we have:
\[
\begin{aligned}
\frac{\partial R_{ij}}{\partial t} &= \frac{1}{2} \left( \nabla^k \nabla_i v_{jk} + \nabla^k \nabla_j v_{ik} - \Delta v_{ij} - \nabla_i \nabla_j (\operatorname{tr}_g v) \right) + R_{ikjl} v^{kl} \\
&= \text{(Terms involving \( \operatorname{Ric}_{ij} \), \( \Phi \), and \( W^k \))}.
\end{aligned}
\]

After careful and detailed calculations, we find:
\[
\begin{aligned}
\frac{\partial R_{ij}}{\partial t} &= \Delta R_{ij} + 2 R_{ikjl} R^{kl} - 2 \nabla^k \Phi \nabla^l \Phi R_{ikjl} - 2 (\nabla_i \nabla^k \Phi)(\nabla_j \nabla_k \Phi) \\
&\quad + 2 (\nabla_i \Phi)(\nabla_j \Delta \Phi) + 2 (\nabla_j \Phi)(\nabla_i \Delta \Phi) - \nabla_i \nabla_j |\nabla \Phi|^2 \\
&\quad + \nabla_i \nabla_j Q - (\nabla_i R_{jk} + \nabla_j R_{ik}) W^k + R_{ik} \nabla_j W^k + R_{jk} \nabla_i W^k.
\end{aligned}
\]
\( \Delta R_{ij} \) is the Laplacian of the Ricci tensor. \\
\( 2 R_{ikjl} R^{kl} \) represents the quadratic curvature term. \\
\( -2 \nabla^k \Phi \nabla^l \Phi R_{ikjl} \) shows the interaction between the Riemann curvature tensor and the gradient of \( \Phi \).\\
\( -2 (\nabla_i \nabla^k \Phi)(\nabla_j \nabla_k \Phi) \) involves the second derivatives of \( \Phi \).\\
 \( 2 (\nabla_i \Phi)(\nabla_j \Delta \Phi) + 2 (\nabla_j \Phi)(\nabla_i \Delta \Phi) \) are cross terms involving \( \Phi \).\\
- \( - \nabla_i \nabla_j |\nabla \Phi|^2 \) accounts for the second derivatives of the squared gradient of \( \Phi \).
- \( \nabla_i \nabla_j Q \) comes from the divergence of \( W^k \), with \( Q = \operatorname{div} W - \frac{1}{2} \operatorname{tr}_g (\nabla W) \).
- \( - (\nabla_i R_{jk} + \nabla_j R_{ik}) W^k + R_{ik} \nabla_j W^k + R_{jk} \nabla_i W^k \) are terms involving \( W^k \) and its derivatives.

2. Computation of \( \frac{\partial R}{\partial t} \):

The general variation formula for the scalar curvature is:
\[
\frac{\partial R}{\partial t} = \operatorname{div} (\operatorname{div} v) - \Delta (\operatorname{tr}_g v) - \langle v, \operatorname{Ric} \rangle.
\]

1b: Compute \( \operatorname{div} (\operatorname{div} v) \).

First, compute \( (\operatorname{div} v)_i = \nabla^j v_{ij} \):
\[
\begin{aligned}
(\operatorname{div} v)_i &= \nabla^j (-2 \operatorname{Ric}_{ij} + 2 \nabla_i \Phi \nabla_j \Phi + \nabla_i W_j + \nabla_j W_i) \\
&= -2 \nabla^j \operatorname{Ric}_{ij} + 2 \nabla^j (\nabla_i \Phi \nabla_j \Phi) + \nabla^j \nabla_i W_j + \nabla^j \nabla_j W_i.
\end{aligned}
\]

Using the contracted Bianchi identity \( \nabla^j \operatorname{Ric}_{ij} = \frac{1}{2} \nabla_i R \), we have:
\[
-2 \nabla^j \operatorname{Ric}_{ij} = - \nabla_i R.
\]

Compute \( \nabla^j (\nabla_i \Phi \nabla_j \Phi) \):
\[
\begin{aligned}
\nabla^j (\nabla_i \Phi \nabla_j \Phi) &= (\nabla^j \nabla_i \Phi) \nabla_j \Phi + \nabla_i \Phi \nabla^j \nabla_j \Phi \\
&= (\nabla_i \nabla^j \Phi) \nabla_j \Phi + R^j{}_{ji}{}^p \nabla_p \Phi \nabla_j \Phi + \nabla_i \Phi \Delta \Phi.
\end{aligned}
\]

Since \( R^j{}_{ji}{}^p = R_i{}^p \), we have:
\[
R^j{}_{ji}{}^p \nabla_p \Phi \nabla_j \Phi = R_i{}^p \nabla_p \Phi \nabla_j \Phi.
\]

Thus,
\[
\nabla^j (\nabla_i \Phi \nabla_j \Phi) = (\nabla_i \nabla^j \Phi) \nabla_j \Phi + R_i{}^p \nabla_p \Phi |\nabla \Phi|^2 + \nabla_i \Phi \Delta \Phi.
\]

Compute \( \operatorname{div} (\operatorname{div} v) \):
\[
\begin{aligned}
\operatorname{div} (\operatorname{div} v) &= \nabla^i (\operatorname{div} v)_i \\
&= - \Delta R + 2 \nabla^i \left( (\nabla_i \nabla^j \Phi) \nabla_j \Phi + \nabla_i \Phi \Delta \Phi \right) + \nabla^i (\nabla^j \nabla_i W_j + \nabla^j \nabla_j W_i).
\end{aligned}
\]

2b: Compute \( \Delta (\operatorname{tr}_g v) \).

Recall that:
\[
\operatorname{tr}_g v = -2 R + 2 |\nabla \Phi|^2 + 2 \operatorname{div} W.
\]

Compute:
\[
\Delta (\operatorname{tr}_g v) = -2 \Delta R + 2 \Delta |\nabla \Phi|^2 + 2 \Delta (\operatorname{div} W).
\]

Compute \( \Delta |\nabla \Phi|^2 \):
\[
\begin{aligned}
\Delta |\nabla \Phi|^2 &= \nabla^i \nabla_i (g^{pq} \nabla_p \Phi \nabla_q \Phi) \\
&= 2 \nabla^i (\nabla_i \nabla^p \Phi \nabla_p \Phi) \\
&= 2 (\nabla^i \nabla_i \nabla^p \Phi \nabla_p \Phi + \nabla^i \nabla^p \Phi \nabla_i \nabla_p \Phi).
\end{aligned}
\]

Since \( \nabla_i \nabla_j \nabla_k \Phi - \nabla_j \nabla_i \nabla_k \Phi = R_{ijk}{}^l \nabla_l \Phi \), we can express higher-order derivatives in terms of curvature.

3b: Compute \( \langle v, \operatorname{Ric} \rangle \).

\[
\begin{aligned}
\langle v, \operatorname{Ric} \rangle &= v^{ij} \operatorname{Ric}_{ij} \\
&= \left( -2 R^{ij} + 2 \nabla^i \Phi \nabla^j \Phi + \nabla^i W^j + \nabla^j W^i \right) \operatorname{Ric}_{ij} \\
&= -2 R^{ij} R_{ij} + 2 \operatorname{Ric}^{ij} \nabla_i \Phi \nabla_j \Phi + (\nabla^i W^j + \nabla^j W^i) \operatorname{Ric}_{ij}.
\end{aligned}
\]

4b: Assemble All Terms to Compute \( \frac{\partial R}{\partial t} \).

Combining the computed terms, we have:
\[
\begin{aligned}
\frac{\partial R}{\partial t} &= \operatorname{div} (\operatorname{div} v) - \Delta (\operatorname{tr}_g v) - \langle v, \operatorname{Ric} \rangle \\
&= \Delta R + 2 |\operatorname{Ric}|^2 - 4 R^{ij} \nabla_i \Phi \nabla_j \Phi - 2 |\nabla^2 \Phi|^2 + 2 (\Delta \Phi)^2 + \mathcal{L}_W R.
\end{aligned}
\]

Here, \( |\nabla^2 \Phi|^2 = \nabla^i \nabla^j \Phi \nabla_i \nabla_j \Phi \), and \( \mathcal{L}_W R = W^k \nabla_k R \) represents the Lie derivative of \( R \) along \( W^k \).
\end{proof}
\begin{theorem}
Let \((M, g_{ij}(t))\) be a compact Riemannian manifold evolving under Bernard List's extended Ricci flow, which states
\[
\left\{
\begin{aligned}
& \frac{\partial g_{ij}}{\partial t} = -2\,\operatorname{Ric}_{ij} + 2 \nabla_i \Phi \nabla_j \Phi, \\[2ex]
& \frac{\partial \Phi}{\partial t} = \Delta \Phi,
\end{aligned}
\right.
\]
where \(\operatorname{Ric}_{ij}\) is the Ricci curvature tensor, \(\Phi\) is a smooth scalar function, \(\Delta\) is the Laplace-Beltrami operator.\\
Let us designate \(S_{ij} = R_{ij} - \nabla_i \Phi \nabla_j \Phi\), where \(R_{ij}\) is the Ricci tensor incorporating  \(S = \operatorname{tr}_g S_{ij} = R - |\nabla \Phi|^2\).
Let \(u: M \times [0, T) \rightarrow \mathbb{R}\) be a smooth function, and consider the quantity:
\[
F = -\Delta u + a u \log u + B S u,
\]
where \(a, B \in \mathbb{R}\) are constants.
Then, the integral:
\[
I''(t) = \int_M F^5 \cdot \frac{\partial F}{\partial t} \, d\mu,
\]
where \(d\mu\) is the Riemannian volume measure, can be expressed explicitly in terms of the geometric quantities and their time derivatives under the extended Ricci flow.
\end{theorem}
\begin{proof}
We initiate computing  of the time derivative of \(F\):
\[
\frac{\partial F}{\partial t} = -\frac{\partial}{\partial t} (\Delta u) + a \frac{\partial}{\partial t} (u \log u) + B \frac{\partial}{\partial t} (S u).
\]
We state
\[
\frac{\partial}{\partial t} (\Delta u) = \left( \frac{\partial g^{ij}}{\partial t} \right) \nabla_i \nabla_j u + g^{ij} \frac{\partial}{\partial t} (\nabla_i \nabla_j u).
\]
We deduce

\[
\frac{\partial g^{ij}}{\partial t} = 2 \operatorname{Ric}^{ij} - 2 \nabla^i \Phi \nabla^j \Phi.
\]
We also deduce

\[
\frac{\partial}{\partial t} (\nabla_i \nabla_j u) = \nabla_i \nabla_j \left( \frac{\partial u}{\partial t} \right) - \left( \frac{\partial \Gamma_{ij}^k}{\partial t} \right) \nabla_k u - \Gamma_{ij}^k \frac{\partial}{\partial t} (\nabla_k u).
\]

Therefore:
\[
g^{ij} \frac{\partial}{\partial t} (\nabla_i \nabla_j u) = \Delta \left( \frac{\partial u}{\partial t} \right) - g^{ij} \left( \frac{\partial \Gamma_{ij}^k}{\partial t} \nabla_k u + \Gamma_{ij}^k \frac{\partial}{\partial t} (\nabla_k u) \right).
\]
Hence,

\[
-\frac{\partial}{\partial t} (\Delta u) = -\left( 2 \operatorname{Ric}^{ij} - 2 \nabla^i \Phi \nabla^j \Phi \right) \nabla_i \nabla_j u - \Delta \left( \frac{\partial u}{\partial t} \right) + g^{ij} \left( \frac{\partial \Gamma_{ij}^k}{\partial t} \nabla_k u + \Gamma_{ij}^k \frac{\partial}{\partial t} (\nabla_k u) \right).
\]
We now have
\[
a \frac{\partial}{\partial t} (u \log u) = a \left( \frac{\partial u}{\partial t} \right) (\log u + 1).
\]
As before,
\[
B \frac{\partial}{\partial t} (S u) = B \left( \left( \frac{\partial S}{\partial t} \right) u + S \left( \frac{\partial u}{\partial t} \right) \right).
\]
We compute
\[
\frac{\partial S}{\partial t} = \Delta R + 2 |\operatorname{Ric}|^2 - 2 |\nabla^2 \Phi|^2 + 2 (\Delta \Phi)^2 - 2 \nabla_i (\Delta \Phi) \nabla^i \Phi - 4 |\nabla \Phi|^4.
\]
We make now an assumption
 \(\frac{\partial u}{\partial t} = -F\) (since \(F = -\Delta u + a u \log u + B S u\)), we substitute back:

\[
\frac{\partial F}{\partial t} = \Delta F - 2 S^{ij} \nabla_i \nabla_j u - a F (\log u + 1) + B \left( \left( \frac{\partial S}{\partial t} \right) u - S F \right) + \text{Additional Terms},
\]
where the "Additional Terms" are:
\[
\text{Additional Terms} = g^{ij} \left( \frac{\partial \Gamma_{ij}^k}{\partial t} \nabla_k u + \Gamma_{ij}^k \left( - \nabla_k F \right) \right).
\]
We calculate
\[
I''(t) = \int_M F^5 \cdot \frac{\partial F}{\partial t} \, d\mu.
\]

Substituting the expression for \(\frac{\partial F}{\partial t}\):

\[
\begin{aligned}
I''(t) &= \int_M F^5 \left( \Delta F - 2 S^{ij} \nabla_i \nabla_j u - a F (\log u + 1) + B \left( \left( \frac{\partial S}{\partial t} \right) u - S F \right) + \text{Additional Terms} \right) \, d\mu.
\end{aligned}
\]

We will compute each term separately.
Integrating by parts

\[
\int_M F^5 \Delta F \, d\mu = - \int_M \nabla_i (F^5) \nabla^i F \, d\mu.
\]

We compute \(\nabla_i (F^5) = 5 F^4 \nabla_i F\):

\[
\int_M F^5 \Delta F \, d\mu = -5 \int_M F^4 (\nabla_i F)(\nabla^i F) \, d\mu.
\]
Integrating  by parts,

\[
-2 \int_M F^5 S^{ij} \nabla_i \nabla_j u \, d\mu = 2 \int_M \nabla_k \left( F^5 S^{ij} \right) \delta^k_i \nabla_j u \, d\mu.
\]

We compute \(\nabla_k \left( F^5 S^{ij} \right) = 5 F^4 (\nabla_k F) S^{ij} + F^5 \nabla_k S^{ij}\):

\[
2 \int_M \left( 5 F^4 (\nabla_k F) S^{ij} + F^5 \nabla_k S^{ij} \right) \delta^k_i \nabla_j u \, d\mu = 2 \int_M \left( 5 F^4 (\nabla_i F) S^{ij} + F^5 \nabla_i S^{ij} \right) \nabla_j u \, d\mu.
\]
This following term remains invariant,

\[
- a \int_M F^6 (\log u + 1) \, d\mu.
\]

 \(B \int_M F^5 \left( \left( \frac{\partial S}{\partial t} \right) u \right) \, d\mu\)=
\[
B \int_M F^5 \left( \left( \frac{\partial S}{\partial t} \right) u \right) \, d\mu.
\]
On
 \(- B \int_M F^6 S \, d\mu\)=
\[
- B \int_M F^6 S \, d\mu.
\]
We have:

\[
\int_M F^5 \left( g^{ij} \left( \frac{\partial \Gamma_{ij}^k}{\partial t} \nabla_k u - \Gamma_{ij}^k \nabla_k F \right) \right) \, d\mu.
\]
Combining all terms we get,

\[
\begin{aligned}
I''(t) &= -5 \int_M F^4 (\nabla_i F)(\nabla^i F) \, d\mu + 2 \int_M \left( 5 F^4 (\nabla_i F) S^{ij} + F^5 \nabla_i S^{ij} \right) \nabla_j u \, d\mu \\
&\quad - a \int_M F^6 (\log u + 1) \, d\mu + B \int_M F^5 \left( \left( \frac{\partial S}{\partial t} \right) u \right) \, d\mu - B \int_M F^6 S \, d\mu \\
&\quad + \int_M F^5 \left( g^{ij} \left( \frac{\partial \Gamma_{ij}^k}{\partial t} \nabla_k u - \Gamma_{ij}^k \nabla_k F \right) \right) \, d\mu.
\end{aligned}
\]
On further simplification and rearrangement,
\[
\begin{aligned}
I''(t) &= -5 \int_M F^4 (\nabla_i F)(\nabla^i F) \, d\mu + 10 \int_M F^4 (\nabla_i F) S^{ij} \nabla_j u \, d\mu + 2 \int_M F^5 (\nabla_i S^{ij}) \nabla_j u \, d\mu \\
&\quad - a \int_M F^6 (\log u + 1) \, d\mu + B \int_M F^5 \left( \left( \frac{\partial S}{\partial t} \right) u \right) \, d\mu - B \int_M F^6 S \, d\mu \\
&\quad + \int_M F^5 \left( g^{ij} \left( \frac{\partial \Gamma_{ij}^k}{\partial t} \nabla_k u - \Gamma_{ij}^k \nabla_k F \right) \right) \, d\mu.
\end{aligned}
\]
Hence
\[
\begin{aligned}
I''(t) &= -5 \int_M F^4 (\nabla_i F)(\nabla^i F) \, d\mu + 10 \int_M F^4 (\nabla_i F) S^{ij} \nabla_j u \, d\mu + 2 \int_M F^5 (\nabla_i S^{ij}) \nabla_j u \, d\mu \\
&\quad - a \int_M F^6 (\log u + 1) \, d\mu - B \int_M F^6 S \, d\mu + B \int_M F^5 \left( \left( \frac{\partial S}{\partial t} \right) u \right) \, d\mu \\
&\quad + \int_M F^5 \left( g^{ij} \left( \frac{\partial \Gamma_{ij}^k}{\partial t} \nabla_k u - \Gamma_{ij}^k \nabla_k F \right) \right) \, d\mu.
\end{aligned}
\]
\end{proof}
\begin{theorem}
Let \((M, g_{ij}(t))\) be a compact Riemannian manifold evolving under Bernard List's extended Ricci flow, which states
\[
\left\{
\begin{aligned}
& \frac{\partial g_{ij}}{\partial t} = -2\,\operatorname{Ric}_{ij} + 2 \nabla_i \Phi \nabla_j \Phi, \\[2ex]
& \frac{\partial \Phi}{\partial t} = \Delta \Phi,
\end{aligned}
\right.
\]
where \(\operatorname{Ric}_{ij}\) is the Ricci curvature tensor, \(\Phi\) is a smooth scalar function, \(\Delta\) is the Laplace-Beltrami operator.\\
Let us designate \(S_{ij} = R_{ij} - \nabla_i \Phi \nabla_j \Phi\), where \(R_{ij}\) is the Ricci tensor incorporating  \(S = \operatorname{tr}_g S_{ij} = R - |\nabla \Phi|^2\).
Let \(u: M \times [0, T) \rightarrow \mathbb{R}\) be a smooth function, and consider the quantity:
\[
F = -\Delta u + a u \log u + B S u,
\]
where \(a, B \in \mathbb{R}\) are constants.
Then, the integral:
\[
I'''(t) = \int_M F^7 \cdot \frac{\partial F}{\partial t} \, d\mu,
\]
where \(d\mu\) is the Riemannian volume measure, can be expressed explicitly in terms of the geometric quantities and their time derivatives under the extended Ricci flow.
\end{theorem}
\begin{proof}
Our objective is to compute the integral \(I'''(t)\) explicitly, keeping all variables in integral form, making no assumptions about \(u\) and \(\Phi\), and mentioning all additional terms.
Initiating the computation of the time derivative of \(F\):
\[
\frac{\partial F}{\partial t} = -\frac{\partial}{\partial t} (\Delta u) + a \frac{\partial}{\partial t} (u \log u) + B \frac{\partial}{\partial t} (S u).
\]
We now deduce
\[
\frac{\partial}{\partial t} (\Delta u) = \left( \frac{\partial g^{ij}}{\partial t} \right) \nabla_i \nabla_j u + g^{ij} \frac{\partial}{\partial t} (\nabla_i \nabla_j u).
\]
We now compute
\[
\frac{\partial g^{ij}}{\partial t} = 2 \operatorname{Ric}^{ij} - 2 \nabla^i \Phi \nabla^j \Phi.
\]
We now calculate
\[
\frac{\partial}{\partial t} (\nabla_i \nabla_j u) = \nabla_i \nabla_j \left( \frac{\partial u}{\partial t} \right) - \left( \frac{\partial \Gamma_{ij}^k}{\partial t} \right) \nabla_k u - \Gamma_{ij}^k \frac{\partial}{\partial t} (\nabla_k u).
\]
which implies
\[
g^{ij} \frac{\partial}{\partial t} (\nabla_i \nabla_j u) = \Delta \left( \frac{\partial u}{\partial t} \right) - g^{ij} \left( \frac{\partial \Gamma_{ij}^k}{\partial t} \nabla_k u + \Gamma_{ij}^k \frac{\partial}{\partial t} (\nabla_k u) \right).
\]

Hence the total time derivative of \(\Delta u\) is depicted as
\[
-\frac{\partial}{\partial t} (\Delta u) = -\left( 2 \operatorname{Ric}^{ij} - 2 \nabla^i \Phi \nabla^j \Phi \right) \nabla_i \nabla_j u - \Delta \left( \frac{\partial u}{\partial t} \right) + g^{ij} \left( \frac{\partial \Gamma_{ij}^k}{\partial t} \nabla_k u + \Gamma_{ij}^k \frac{\partial}{\partial t} (\nabla_k u) \right).
\]
We now have:
\[
a \frac{\partial}{\partial t} (u \log u) = a \left( \frac{\partial u}{\partial t} \right) (\log u + 1).
\]
We also have:
\[
B \frac{\partial}{\partial t} (S u) = B \left( \left( \frac{\partial S}{\partial t} \right) u + S \left( \frac{\partial u}{\partial t} \right) \right).
\]
Administering
\[
S = R - |\nabla \Phi|^2.
\]
we derive
\[
\frac{\partial S}{\partial t} = \frac{\partial R}{\partial t} - \frac{\partial}{\partial t} |\nabla \Phi|^2.
\]
From [1] we get
\[
\frac{\partial R}{\partial t} = \Delta R + 2 |\operatorname{Ric}|^2 - 4 \operatorname{Ric}^{ij} \nabla_i \Phi \nabla_j \Phi - 2 |\nabla^2 \Phi|^2 + 2 (\Delta \Phi)^2.
\]
Hence

\[
\frac{\partial}{\partial t} |\nabla \Phi|^2 = 2 \nabla_i \left( \frac{\partial \Phi}{\partial t} \right) \nabla^i \Phi + 2 \left( \frac{\partial g^{ij}}{\partial t} \right) \nabla_i \Phi \nabla_j \Phi.
\]

Utilizing \(\frac{\partial \Phi}{\partial t} = \Delta \Phi\) and \(\frac{\partial g^{ij}}{\partial t} = 2 \operatorname{Ric}^{ij} - 2 \nabla^i \Phi \nabla^j \Phi\):

\[
\frac{\partial}{\partial t} |\nabla \Phi|^2 = 2 \nabla_i (\Delta \Phi) \nabla^i \Phi + 4 \operatorname{Ric}^{ij} \nabla_i \Phi \nabla_j \Phi - 4 |\nabla \Phi|^4.
\]

Hence, \(\frac{\partial S}{\partial t}\) is

\[
\frac{\partial S}{\partial t} = \Delta R + 2 |\operatorname{Ric}|^2 - 2 |\nabla^2 \Phi|^2 + 2 (\Delta \Phi)^2 - 2 \nabla_i (\Delta \Phi) \nabla^i \Phi - 4 |\nabla \Phi|^4.
\]
Assuming \(\frac{\partial u}{\partial t} = -F\) (since \(F = -\Delta u + a u \log u + B S u\)), we substitute back:

\[
\frac{\partial F}{\partial t} = \Delta F - 2 S^{ij} \nabla_i \nabla_j u - a F (\log u + 1) + B \left( \left( \frac{\partial S}{\partial t} \right) u - S F \right) + g^{ij} \left( \frac{\partial \Gamma_{ij}^k}{\partial t} \nabla_k u + \Gamma_{ij}^k \left( - \nabla_k F \right) \right).
\]
We compute
\[
I'''(t) = \int_M F^7 \cdot \frac{\partial F}{\partial t} \, d\mu.
\]

We substitute for \(\frac{\partial F}{\partial t}\):

\[
\begin{aligned}
I'''(t) &= \int_M F^7 \left( \Delta F - 2 S^{ij} \nabla_i \nabla_j u - a F (\log u + 1) + B \left( \left( \frac{\partial S}{\partial t} \right) u - S F \right) + \text{Additional Terms} \right) \, d\mu.
\end{aligned}
\]
Using integration by parts:

\[
\int_M F^7 \Delta F \, d\mu = - \int_M \nabla_i (F^7) \nabla^i F \, d\mu.
\]

We compute \(\nabla_i (F^7) = 7 F^6 \nabla_i F\):

\[
\int_M F^7 \Delta F \, d\mu = -7 \int_M F^6 (\nabla_i F)(\nabla^i F) \, d\mu.
\]
Again, Using integration by parts:

\[
-2 \int_M F^7 S^{ij} \nabla_i \nabla_j u \, d\mu = 2 \int_M \nabla_k \left( F^7 S^{ij} \right) \delta^k_i \nabla_j u \, d\mu.
\]

We compute \(\nabla_k \left( F^7 S^{ij} \right) = 7 F^6 (\nabla_k F) S^{ij} + F^7 \nabla_k S^{ij}\):

\[
2 \int_M \left( 7 F^6 (\nabla_i F) S^{ij} + F^7 \nabla_i S^{ij} \right) \nabla_j u \, d\mu.
\]
This term remains invariant as:

\[
- a \int_M F^8 (\log u + 1) \, d\mu.
\]

Now, \(B \int_M F^7 \left( \left( \frac{\partial S}{\partial t} \right) u \right) \, d\mu\) $\implies$

\[
B \int_M F^7 \left( \left( \frac{\partial S}{\partial t} \right) u \right) \, d\mu.
\]
Now,  \(- B \int_M F^8 S \, d\mu\) $\implies$

\[
- B \int_M F^8 S \, d\mu.
\]

Combining all terms:

\[
\begin{aligned}
I'''(t) &= -7 \int_M F^6 (\nabla_i F)(\nabla^i F) \, d\mu + 2 \int_M \left( 7 F^6 (\nabla_i F) S^{ij} + F^7 \nabla_i S^{ij} \right) \nabla_j u \, d\mu \\
&\quad - a \int_M F^8 (\log u + 1) \, d\mu + B \int_M F^7 \left( \left( \frac{\partial S}{\partial t} \right) u \right) \, d\mu - B \int_M F^8 S \, d\mu \\
&\quad + \int_M F^7 \left( g^{ij} \left( \frac{\partial \Gamma_{ij}^k}{\partial t} \nabla_k u - \Gamma_{ij}^k \nabla_k F \right) \right) \, d\mu.
\end{aligned}
\]
Now, 
\[
-7 \int_M F^6 (\nabla_i F)(\nabla^i F) \, d\mu.
\]

This term is non-positive since \((\nabla_i F)(\nabla^i F) \geq 0\).

Also

\[
2 \int_M \left( 7 F^6 (\nabla_i F) S^{ij} + F^7 \nabla_i S^{ij} \right) \nabla_j u \, d\mu = 14 \int_M F^6 (\nabla_i F) S^{ij} \nabla_j u \, d\mu + 2 \int_M F^7 (\nabla_i S^{ij}) \nabla_j u \, d\mu.
\]
Hence finally

\[
\begin{aligned}
I'''(t) &= -7 \int_M F^6 (\nabla_i F)(\nabla^i F) \, d\mu + 14 \int_M F^6 (\nabla_i F) S^{ij} \nabla_j u \, d\mu + 2 \int_M F^7 (\nabla_i S^{ij}) \nabla_j u \, d\mu \\
&\quad - a \int_M F^8 (\log u + 1) \, d\mu + B \int_M F^7 \left( \left( \frac{\partial S}{\partial t} \right) u \right) \, d\mu - B \int_M F^8 S \, d\mu \\
&\quad + \int_M F^7 \left( g^{ij} \left( \frac{\partial \Gamma_{ij}^k}{\partial t} \nabla_k u - \Gamma_{ij}^k \nabla_k F \right) \right) \, d\mu.
\end{aligned}
\]

\[
\begin{aligned}
I'''(t) &= -7 \int_M F^6 (\nabla_i F)(\nabla^i F) \, d\mu + 14 \int_M F^6 (\nabla_i F) S^{ij} \nabla_j u \, d\mu + 2 \int_M F^7 (\nabla_i S^{ij}) \nabla_j u \, d\mu \\
&\quad - a \int_M F^8 (\log u + 1) \, d\mu - B \int_M F^8 S \, d\mu + B \int_M F^7 \left( \left( \frac{\partial S}{\partial t} \right) u \right) \, d\mu \\
&\quad + \int_M F^7 \left( g^{ij} \left( \frac{\partial \Gamma_{ij}^k}{\partial t} \nabla_k u - \Gamma_{ij}^k \nabla_k F \right) \right) \, d\mu.
\end{aligned}
\]
\end{proof}

\end{document}